\documentclass{amsart}

\usepackage{amsmath,amsfonts,amssymb}
\usepackage[colorlinks=true]{hyperref}
\usepackage{stmaryrd}
\usepackage[all]{xy}

\newtheorem*{theorem*}{Theorem} 
\newtheorem{theorem}{Theorem}[section]
\newtheorem{prop}[theorem]{Proposition}
\newtheorem{lemma}[theorem]{Lemma}
\newtheorem{cor}[theorem]{Corollary}

\theoremstyle{definition}
\newtheorem{definition}[theorem]{Definition}
\newtheorem{rem}[theorem]{Remark}

\newcommand{\ie}{{\it i.e.}\ }

\newcommand{\IC}{\mathbb{C}}

\newcommand{\IG}{\mathbb{G}}
\newcommand{\IL}{\mathbb{L}}

\newcommand{\IP}{\mathbb{P}}

\newcommand{\IR}{\mathbb{R}}
\newcommand{\IZ}{\mathbb{Z}}

\newcommand{\cD}{\mathcal{D}}
\newcommand{\cF}{\mathcal{F}}

\newcommand{\cM}{\mathcal{M}}

\newcommand{\cO}{\mathcal{O}}
\newcommand{\cP}{\mathcal{P}}
\newcommand{\cX}{\mathcal{X}}

\newcommand{\HH}{{\rm H}}

\DeclareMathOperator{\PGL}{PGL}
\DeclareMathOperator{\Span}{Span}
\DeclareMathOperator{\Aut}{Aut}

\DeclareMathOperator{\Mon}{Mon}
\DeclareMathOperator{\NS}{NS}

\DeclareMathOperator{\Orth}{O}
\DeclareMathOperator{\Inv}{T}
\DeclareMathOperator{\Def}{Def}
\DeclareMathOperator{\rank}{rank}

\newcommand{\id}{{\rm id}}

\newcommand{\KKK}{{\rm K3}}

\newcommand{\vol}{{\rm vol}}

\newcommand{\ample}{{\rm ample}}

\newcommand{\coloneqq}{:=}

\title[Involutions of the Hilbert square of a K3 surface]{On the nonsymplectic involutions\\ of the Hilbert square of a K3 surface}

\date{\today}

\author{}

\author[S. Boissi\`ere]{Samuel Boissi\`ere}
\address{Samuel Boissi\`ere, Universit\'e de Poitiers, 
Laboratoire de Math\'ematiques et Applications, 
UMR CNRS 7348,
T\'el\'eport 2 
Boulevard Marie et Pierre Curie, 
BP 30179,
86962 Futuroscope Chasseneuil Cedex, France}
\email{samuel.boissiere@math.univ-poitiers.fr}
\urladdr{http://www-math.sp2mi.univ-poitiers.fr/$\sim$sboissie/}

\author[A. Cattaneo]{Andrea Cattaneo}
\address{Andrea Cattaneo, Institut Camille Jordan, UMR 5208, Universit\'e Claude Bernard Lyon 1, 69622 Villeurbanne Cedex, France}
\email{cattaneo@math.univ-lyon1.fr}
\thanks{Andrea Cattaneo is supported by the LABEX MILYON (ANR-10-LABX-0070) of Universit\'e de Lyon, within the program ``Investissements d'Avenir'' (ANR-11-IDEX-0007) operated by the French National Research Agency (ANR)}

\author[D. Markushevich]{Dimitri Markushevich}
\address{Dimitri Markuchevitch, Universit\'e des Sciences et Technologies de Lille 1, Laboratoire Paul Painlev\'e, 
UMR CNRS 8524,
59655 Villeneuve d'Ascq Cedex}
\email{markushe@math.univ-lille1.fr}
\urladdr{}

\author[A. Sarti]{Alessandra Sarti}
\address{Alessandra Sarti, Universit\'e de Poitiers, 
Laboratoire de Math\'ematiques et Applications, 
UMR CNRS 7348,
T\'el\'eport 2 
Boulevard Marie et Pierre Curie,
BP 30179,
86962 Futuroscope Chasseneuil Cedex, France}
\email{sarti@math.univ-poitiers.fr}
\urladdr{http://www-math.sp2mi.univ-poitiers.fr/$\sim$sarti/}


\begin{document}

\begin{abstract}
We investigate the interplay between the moduli spaces of ample $\langle 2\rangle$-polarized IHS manifolds of type $\KKK^{[2]}$ and of IHS manifolds of type $\KKK^{[2]}$ with a nonsymplectic involution with invariant lattice of rank one. In particular we geometrically describe some new involutions of the Hilbert square of a K3 surface, whose existence was proven in a previous work of  Boissi\`ere--Cattaneo--Nieper-Wi{\ss}kirchen--Sarti.
\end{abstract}

\maketitle


\section{Introduction}

By a classical result of Saint-Donat~\cite{saintdonat}, every ample $\langle 2\rangle$-polarized complex K3 surface is a double cover of the complex projective plane branched along a
smooth sextic curve and it admits in a natural way a nonsymplectic involution. The cohomological invariant sublattice is generically
isometric to the rank one lattice $\langle 2\rangle$, generated by the pullback of the class of a line in the plane. The converse is also true: if a K3 surface admits a nonsymplectic involution whose invariant lattice has rank one, then this lattice is isometric to the lattice~$\langle 2\rangle$ and it is generated by an ample class, so the K3 surface can be constructed as a double cover of the plane branched along a smooth sextic curve.

The present paper focuses on the generalization of the above result  for a class of irreducible holomorphic symplectic (IHS) manifolds. 
These can be seen as a higher dimensional generalization of K3 surfaces, and share several properties with them. These are simply connected manifolds with a unique (up to scalar multiplication) holomorphic 2--form which is everywhere nondegenerate.
This implies that their dimension is even and that their canonical divisor is trivial. Moreover their second cohomology with integer coefficients is a lattice for the Beauville--Bogomolov--Fujiki quadratic form. One of the most studied families
of IHS manifolds is the $2n$-dimensional Hilbert scheme of $n$ points on a smooth complex projective K3 surface, which has a $20$--dimensional moduli space. 

The study of automorphisms of IHS manifolds was started by Arnaud Beauville~\cite{B_remarks, beauville1}, who generalized to IHS manifolds several results of Nikulin for automorphisms of K3 surfaces.
In this paper, we focus on \emph{ample $\langle 2 \rangle$-polarized} IHS manifolds of type $\KKK^{[2]}$, \ie deformation equivalent to the Hilbert square of a projective K3 surface (see Section~\ref{s:prel}). The first main result of this paper is Theorem~\ref{th:ample}, which generalizes Saint-Donat's result for K3 surfaces: 

\begin{theorem*}\text{}
\begin{enumerate}
\item Let $(X,D)$ be an ample $\langle 2\rangle$-polarized IHS manifold of type $\KKK^{[2]}$. Then $X$~admits a nonsymplectic involution $\sigma$ whose action on $\HH^2(X,\IZ)$ is the reflection in the class of $D$ in $\HH^2(X,\IZ)$.  
\item Conversely, let $X$ be an IHS manifold of type $\KKK^{[2]}$ with a nonsymplectic involution $\sigma$ whose invariant lattice $\Inv(\sigma)$ has rank $1$. Then $\Inv(\sigma)$ is generated by the class of an ample divisor $D$ of square $2$ and $\sigma$ acts on $\HH^2(X,\IZ)$ as the reflection in the class of $D$.
\end{enumerate} 
\end{theorem*}

In Section~\ref{s:moduli}, we interpret this result as an isomorphism between the moduli space of ample $\langle 2\rangle$--polarized IHS manifolds of type $\KKK^{[2]}$ introduced by Gritsenko--Hulek--Sankaran~\cite{GHS} and recently reconsidered by Debarre--Macr\`i~\cite{DM} on one side, and the moduli space of IHS manifolds of type $\KKK^{[2]}$ with a nonsymplectic involution of invariant lattice of rank one on the other side. Then in Section~\ref{ss:defo} we show that in this setup, two automorphisms can always be deformed to each other, and in particular to a Beauville involution or to an O'Grady involution. The major part of these results where already known by experts but to our knowledge a complete proof was missing.

Section~\ref{s:new} is devoted to geometric constructions of nonsymplectic involutions with invariant lattice isometric to $\langle 2\rangle$. This is in general a difficult problem: in~\cite{bcns} these involutions are classified with the help of the Global Torelli theorem of Markman--Verbitsky~\cite{markman}, but until now the only geometric example was the Beauville involution (see Definition~\ref{def:beauville}). In \cite{bcns} the authors consider a generic projective K3 surface~$S$ with an ample polarization of square $2t$, with $t$ a positive integer, and by using Torelli theorem they show the existence of nonsymplectic involutions
on $S^{[2]}$ for certain values of $t$ (see~Section~\ref{s:new}). The question of constructing geometrically these involutions remained open and it is the second goal of this paper. We consider special K3 surfaces admitting two embeddings as quartics in $\IP^3$ and we show how one can use Beauville involutions to construct the nonsymplectic involution on the Hilbert scheme. 
Finally in Section~\ref{ss:nodesK3} we use nodal K3 surfaces to give different geometric constructions of Beauville involutions.

\subsection*{Acknowledgements} Part of the work was done during the 2015 Oberwolfach mini--workshop {\it Singular Curves on K3 Surfaces and Hyperk\"ahler Manifolds}. The authors thank this institution for the stimulating working atmosphere.


\section{Preliminary notions}\label{s:prel}

A {\it lattice} $L$ is a free $\IZ$-module endowed with an integer valued, nondegenerate, symmetric bilinear form. A sublattice $M\subset L$ is called \emph{primitive} if $L/M$ is a free $\IZ$-module.

A compact complex K\"{a}hler manifold $ X $ is called {\it irreducible holomorphic symplectic} (IHS) if it is simply connected and if it admits a holomorphic  $2$-form~$\omega_X$ everywhere nondegenerate and unique up to scalar multiplication.
The existence of such a symplectic form immediately implies that the dimension of $X $ is even. 
The second cohomology group with integer coefficients $\HH^2(X,\IZ)$ is a lattice for the Beauville--Bogomolov--Fujiki~\cite{beauville1} quadratic form $q_X$. We denote by $\langle-,-\rangle_X$ the associated bilinear form.

The group $\Aut(X)$  of biholomorphic automorphisms of $X$ is a discrete group. An element $\sigma\in\Aut(X)$ is called \emph{symplectic} if $\sigma^\ast\omega_X=\omega_X$, and \emph{nonsymplectic} otherwise. In particular a nonsymplectic involution $\sigma$ is such that $\sigma^\ast\omega_X=-\omega_X$. The \emph{invariant lattice} of $\sigma\in\Aut(X)$ is the primitive sublattice $\Inv(\sigma)\subset \HH^2(X,\IZ)$ consisting  of the cohomology classes invariant by $\sigma^\ast$.

An IHS manifold $X$ is called \emph{of type $\KKK^{[2]}$} if it is deformation equivalent to the Hilbert scheme (or Douady space if nonalgebraic) of $0$-dimensional subschemes of length $2$ of  a  K3 surface. In this case, the lattice $(\HH^2(X,\mathbb{Z}),q_X)$ is isometric to
the lattice: 
$$
L\coloneqq U^{\oplus 3}\oplus E_8(-1)^{\oplus 2}\oplus \langle -2\rangle
$$
(by convention $U$ is the hyperbolic plane and the root lattice $E_8$ is positive definite).

The \emph{monodromy group} $\Mon^2(X)$ is the subgroup of $\Orth(\HH^2(X,\IZ))$ generated by the image of all monodromy representations of smooth proper holomorphic families with central fiber $X$. 
By Markman~\cite[Theorem 1.2]{Markmanintegral}, if $X$ is of type $\KKK^{[2]}$ then $\Mon^2(X)$ is equal to the subgroup $\Orth^+(\HH^2(X,\IZ))$ of isometries whose real extension preserves the orientation of any positive definite $3$-dimensional subspace of $\HH^2(X,\IR)$. In this context, a natural orientation of $X$ is  $\Span(\Re(\omega_X),\Im(\omega_X),\kappa)$ where $\kappa$~is a K\"ahler class. More generally, if $X$ and $Y$ are IHS manifolds in the same deformation class, an isometry $\varphi\colon \HH^2(X,\IZ)\to\HH^2(Y,\IZ)$ is called \emph{orientation-preserving} if it respects the orientation of any pair of positive definite $3$-dimensional subspaces of $\HH^2(X,\IZ)$ and $\HH^2(Y,\IZ)$. 

An \emph{ample $\langle 2\rangle$-polarized} IHS manifold is by definition a pair $(X,D)$ consisting of a projective IHS manifold $X$ and an ample divisor $D$ on $X$ of square~$2$. It has a canonical orientation given by the positive definite $3$-plane $\Span(\Re(\omega_X),\Im(\omega_X),[D])$.

We recall for later use a Hodge theoretic version of the Torelli theorem
in the special case of IHS manifolds of type $\KKK^{[2]}$:

\begin{theorem}\cite[Theorem~1.3,Theorem~9.1,Theorem~9.5,Theorem~9.8]{markman}\label{th:torelli}
Let $X$ and~$Y$ be IHS manifolds of type  $\KKK^{[2]}$ and $\varphi\colon \HH^2(X,\IZ)\to\HH^2(Y,\IZ)$ be an orientation-preserving isometry. Assume that $\varphi$ induces an isomorphism of Hodge structures between $\HH^2(X,\IC)$ and $\HH^2(Y,\IC)$ and that the image by~$\varphi$ of a K\"ahler class of $X$ is a K\"ahler class of $Y$. Then there exists an biregular isomorphism $f\colon X\to Y$ such that $f^\ast=\varphi$.
\end{theorem}

\begin{rem}
Observe that in this setup, a Hodge isometry preserving a K\"ahler class is automatically orientation-preserving since $\omega_X$ is mapped to a complex multiple of~$\omega_Y$.
\end{rem}


\section{Ample polarizations and nonsymplectic involutions}

This section is inspired by the following classical result of Saint-Donat~\cite{saintdonat} for  IHS manifolds of dimension $2$ already recalled in Introduction. Consider an ample $\langle 2\rangle$-polarized K3 surface $(S,D)$, where $D$~is an ample divisor of square $2$. Then the linear system $|D|$ is base point free and defines a double covering $S\to\IP^2$ branched along a smooth sextic curve. The covering involution $\sigma$ is nonsymplectic and it acts on $\HH^2(S,\IZ)$ 
as the reflection in the class of~$D$. If $(S,D)$ is generic in the moduli space of ample $\langle 2\rangle$-polarized K3 surfaces, the invariant sublattice of $\HH^2(S,\IZ)$ for the action of $\sigma$ is isometric to~$\langle 2\rangle$, generated by the divisor~$D$. The converse is true: any double covering of~$\IP^2$ branched
along a smooth sextic curve is an ample $\langle 2\rangle$-polarized K3 surface.
We extend this result as follows:

\begin{theorem}\label{th:ample}\text{}
\begin{enumerate}
\item\label{th:ample_item1} Let $(X,D)$ be an ample $\langle 2\rangle$-polarized IHS manifold of type $\KKK^{[2]}$. Then $X$~admits a nonsymplectic involution $\sigma$ whose action on $\HH^2(X,\IZ)$ is the reflection in the class of $D$ in $\HH^2(X,\IZ)$.  
\item\label{th:ample_item2} Conversely, let $X$ be an IHS manifold of type $\KKK^{[2]}$ with a nonsymplectic involution $\sigma$ whose invariant lattice $\Inv(\sigma)$ has rank $1$. Then $\Inv(\sigma)$ is generated by the class of an ample divisor $D$ of square $2$ and $\sigma$ acts on $\HH^2(X,\IZ)$ as the reflection in the class of $D$.
\end{enumerate} 
\end{theorem}

\begin{proof}\text{}
\begin{enumerate}
\item Consider the reflection  
$$
\varphi\colon\HH^2(X,\IZ)\to \HH^2(X,\IZ), \quad v\mapsto \langle v,[D]\rangle_X \cdot [D]-v,
$$
whose invariant lattice $\IZ[D]$ is isometric to $\langle 2\rangle$. We know that $\varphi\in\Mon^2(X)$ by~\cite[Corollary~1.8]{Markman_monodromyK3}, but this can be checked again as follows using the characterization of $\Mon^2(X)$ recalled above. Since $\HH^{2,0}(X)=\IC\omega_X$ is orthogonal to the algebraic class $[D]$, $\varphi$ acts as multiplication by $(-1)$ on $\HH^{2,0}(X)$, 
so $\varphi$~is a Hodge isometry. Since $\varphi$ leaves the ample class $[D]$ invariant, it preserves the orientation of the positive definite $3$-dimensional 
subspace $\Span(\Re(\omega_X),\Im(\omega_X),[D])$ of $\HH^2(X,\IR)$, so $\varphi\in\Mon^2(X)$. By Theorem~\ref{th:torelli}, there exists $\sigma\in\Aut(X)$
such that $\sigma^*=\varphi$. Since the map $\Aut(X)\to \Orth(\HH^2(X,\IZ))$ is injective (see \cite[Lemma 1.2]{mongardinat}), $\sigma$ is an involution.
\item Looking at the lattice theoretical classification of nonsymplectic involutions acting on IHS manifolds of type $\KKK^{[2]}$ (see~\cite[Proposition 8.2]{bcs}), we know that the invariant lattice $T(\sigma)$ is isometric to the lattice $\langle 2\rangle$, generated by the class of a divisor~$D$ of square~$2$. Since $X$~admits a nonsymplectic automorphism, it is projective (see \cite[Proposition 6]{B_remarks}). Consider an ample class $\ell$ of $X$. The class $\ell+\sigma^*(\ell)$ is ample and $\sigma$-invariant, so it is a multiple of the generator~$[D]$ of the invariant lattice, hence $D$ is ample. Using the first assertion we get that $\sigma$ is the reflection in the class $[D]$.
\end{enumerate}
\end{proof}

\begin{rem}
One can state very similar results for IHS manifolds of type $K3^{[n]}$, $n\geq 3$ (see Alberto Cattaneo~\cite{alberto}).
Compared to the $2$-dimensional situation, in this $4$-dimensional setup the full understanding of the linear system $|D|$ of the ample class of square $2$ remains open (see~\cite{bcns}).
\end{rem}


\section{Modular interpretation}\label{s:moduli}

We interpret the result above as an isomorphism between two moduli spaces. We fix a primitive embedding 
$j\colon\langle 2\rangle\hookrightarrow L$. Such an embedding is unique up to an isometry 
of $L$, see~\cite[Proposition\ 8.2]{bcs}. We denote by $\rho\in\Orth(L)$ the reflection in the ray $j(\langle 2\rangle)$.

By Gritsenko--Hulek--Sankaran~\cite[\S3]{GHS_moduliK3} (see also Debarre--Macr\`i~\cite[\S3.1]{DM}), there exists a quasi-projective coarse moduli space $\cM_{\langle 2\rangle}$, which is irreducible and $20$-dimensional, parametrizing pairs $(X,\iota)$ where 
$X$ is an IHS manifold of type $\KKK^{[2]}$ and $\iota\colon \langle 2\rangle\hookrightarrow\NS(X)$ is a primitive embedding, with an open subspace $\cM_{\langle 2\rangle}^\ample$ parametrizing pairs $(X,\iota)$ such that $\iota(\langle 2\rangle)$ contains the class of an ample divisor.

By results of Verbitsky (see~\cite{DM} and references therein), there is an algebraic period map $\cP_{\langle 2\rangle}\colon\cM_{\langle 2\rangle}\to \cD_{\langle 2\rangle}$, mapping $(X,\iota)$ to its period $\HH^{2,0}(X)$, which is an open embedding, where $\cD_{\langle 2\rangle}$ is an irreducible algebraic variety.

By Boissi\`ere--Camere--Sarti~\cite[Theorem~4.5,Theorem~5.6]{BCS_ball}, there exists a quasi-projective coarse moduli space $\cM_{\langle 2\rangle}^\rho$, which is irreducible and $20$-dimensional, parametrizing triples $(X,\sigma,\iota)$ where 
$X$ is an IHS of type $\KKK^{[2]}$, $\iota\colon \langle 2\rangle\hookrightarrow\NS(X)$ is a primitive embedding and $\sigma\in\Aut(X)$ is  a nonsymplectic involution whose action on $\HH^2(X,\IZ)$ is conjugated to $\rho$.

\begin{cor}\label{prop:modspace}
 The moduli spaces $\cM_{\langle 2\rangle}^\rho$ and $\cM_{\langle 2\rangle}^\ample$ are isomorphic.
\end{cor}

\begin{proof} The natural map $\cM^\rho_{\langle 2\rangle}\to\cM_{\langle 2\rangle}^\ample$ is given by Theorem~\ref{th:ample}(\ref{th:ample_item2}).
Conversely, starting from $(X,\iota)\in  \cM_{\langle 2\rangle}^\ample$, denote by $D$ an ample divisor generating $\iota(\langle 2\rangle)$. By
Theorem~\ref{th:ample}(\ref{th:ample_item1}) the pair $(X,D)$ admits a nonsymplectic involution~$\sigma$ such that $\sigma^\ast$~is the reflection in the ray~$\IZ[D]$. Then $\sigma^\ast$ is conjugated to $\rho$ so $(X,\sigma,\iota)\in\cM_{\langle 2\rangle}^\rho$.
\end{proof}

\section{Deformation of nonsymplectic involutions}

We exhibit two famous families of nonsymplectic involutions. The first one, the \emph{Beauville family}, gives a codimension one subspace of $\cM^\rho_{\langle 2\rangle}$. The second one, the \emph{O'Grady family}, is dense in $\cM^\rho_{\langle 2\rangle}$.

\subsection{The Beauville family}\label{ss:beauville}

Let $S$ be a projective K3 surface. Recall that the N\'eron--Severi group of the Hilbert scheme $S^{[2]}$ of $2$ points on $S$ admits an orthogonal decomposition 
$$
\NS(S^{[2]})\cong\NS(S)\oplus\IZ\delta,
$$ 
where $2\delta$ is the class of the exceptional divisor of $S^{[2]}$, with $q_{S^{[2]}}(\delta)=-2$.

\begin{definition}\label{def:beauville} A \emph{Beauville involution} on $S^{[2]}$ is a nonsymplectic involution~$\sigma$ whose invariant lattice has rank~$1$, generated by the class of an ample divisor~$D$ of square~$2$, 
which decomposes as $[D]=[H]-\delta$ where $H$ is a very ample divisor of square $4$ on $S$.
\end{definition}

Beauville involutions are geometrically realized as follows.
Let $S$~be a smooth quartic K3 surface in $\IP^3$ containing no line
and $|H|$ be the linear system of the hyperplane section. The line $\ell$ through a 
subscheme $\xi\in S^{[2]}$ cuts $S$ in a residual length two subscheme, providing the involution $\sigma$
(see Beauville~\cite{B_remarks} for details). The construction of such nonsymplectic involutions depend on $19$ parameters, corresponding similarly as above to the moduli of ample $\langle 4\rangle$-polarized K3 surfaces.

\subsection{The O'Grady family}

We briefly recall the construction of double EPW sextics~\cite{ogradydouble}. Let $V$ be a $6$-dimensional complex vector space. We choose an isomorphism 
$\vol\colon \bigwedge^6 V \to \IC$, thus making $\bigwedge^3 V$ a symplectic vector space by  $(\alpha, \beta) \coloneqq \vol(\alpha\wedge \beta)$. The $10$-dimensional Lagrangian subspaces 
$$
\cF_v=\left\{v\wedge \alpha \,|\, \alpha\in \bigwedge^2 V\right\}\cong \bigwedge^2 \left(V/\IC v\right),\quad v\in V,
$$
are the fibers of the sheaf $\cF = \cO_{\IP(V)}(-1) \otimes \bigwedge^2 Q$, where $Q$ is the tautological bundle.
By a Chern classes computation one obtains that $c_1(\cF)=-6c_1(\cO_{\IP(V)}(1))$.

Let $A \subseteq \bigwedge^3 V$ be a lagrangian subspace. By inclusion and projection one gets a map of vector bundles:
$$
\lambda_A: \cF\to \bigwedge^3 V\otimes \cO_{\IP(V)}\to \frac{\bigwedge^3 V}{A}\otimes \cO_{\IP(V)},
$$
whose degeneracy locus $Y_A$ is the subscheme of zeros of $\det (\lambda_A)$, which is a global section of the sheaf $\det (\cF^{\vee})=\cO_{\IP(V)}(6)$. 
If $Y_A$ is a proper subscheme of $\IP(V)$,  it is thus a sextic hypersurface called an {\it EPW sextic}~\cite{EPW,ogradydouble}.

Clearly $Y_A = \{ [v] \in \IP(V) \,|\, \dim (\cF_v \cap A) \geq 1 \}$.
For $A$ generic in the Grassmannian variety of Lagrangian subspaces of $\IP(V)$, the fourfold $Y_A$ has only double points on the smooth surface
$W_A \coloneqq\{ [v] \in \IP(V)\,|\, \dim (\cF_v \cap A) \geq 2 \}$.

By O'Grady~\cite[Theorem~1.1]{4foldsEPW}, in this situation $Y_A$ admits a smooth double cover $X_A \to Y_A$ branched along $W_A$, which is an IHS manifold of type $\KKK^{[2]}$ called a \emph{double EPW sextic}. 
By construction,  $X_A$ comes together with its covering involution, which is nonsymplectic, and it leaves invariant an ample class of square~$2$. We call this an \emph{O'Grady involution}. 
This family depends on $20$ parameters, corresponding to the dimension of the GIT quotient $\IL\IG(\wedge^3V)\sslash\PGL(V)$.

\subsection{Deformation equivalences}\label{ss:defo}

Recall that two pairs $(X_1,f_1)$ and $(X_2,f_2)$ of IHS manifolds of the same deformation type, endowed with a nonsymplectic automorphism, are \emph{deformation equivalent} (see~\cite[\S4]{bcs},\cite[Definition~4.5]{joumaah}) if there exists a smooth and proper family $\pi\colon\cX\to\Delta$ of IHS manifolds over a connected smooth analytic space $\Delta$, with a fiber-preserving holomorphic automorphism $F$ which is nonsymplectic on each fiber $\cX_t$, $t\in\Delta$, two points $t_1,t_2\in\Delta$ and isomorphisms $\varphi_i\colon \cX_{t_i}\to X_i$ such that $\varphi_i\circ F_{t_i}=f_i\circ \varphi_i$ for $i=1,2$ (we denote  by $F_t$ the restriction of $F$ on each fiber $\cX_t$).

\begin{theorem}\label{th:defo}
Any two points $(X,i_X),(Y,i_Y)\in\cM_{\langle 2\rangle}^\rho$ are deformation equivalent.
\end{theorem}

\begin{proof}
The statement is an application of a general result of Joumaah~\cite[Theorem~9.10]{joumaah}. However, in our situation it is more enlightening to write down a direct argument following the same idea. To any triple $(X,\sigma,\iota)\in \cM_{\langle 2\rangle}^\rho$ we associate as in~\cite[\S4]{bcs} the $20$-dimensional local deformation space $\cX\to\Def(X,\sigma,\iota)$ endowed with a holomorphic automorphism $F$ of $\cX$ extending $\sigma$ to a nonsymplectic involution on each fiber. The disjoint union $\coprod\limits_{(X,\sigma,\iota)\in\cM_{\langle 2\rangle}^\rho} \Def(X,\sigma,\iota)$ is glued by the equivalence relation given by the period map $\cP_{\langle 2\rangle}\colon\cM_{\langle 2\rangle}\to \cD_{\langle 2\rangle}$. Given two such deformations $F\colon \cX\to\cX$ over a base $U$ and $F'\colon \cX'\to\cX'$ over $U'$, the restrictions of $F$ and $F'$ over the intersection $U\cap U'$ are equal since the period map is injective: the nonsymplectic involution is uniquely determined by the period. So, glueing $\cX$ and $\cX'$ over $U\cap U'$, we can extend $F$ and $F'$ as a holomorphic automorphism over $U\cup U'$, and finally over $\cM_{\langle 2\rangle}$, which is connected, so we get the result.
\end{proof}

The invariant lattice of an automorphism is a topological invariant, thus invariant by deformation. We get as corollary of Theorem~\ref{th:ample} and Theorem~\ref{th:defo}
the following result, which first appeared in Ferretti~\cite[Proposition 4.1]{ferretti}, but the point of view developed in this paper provides a more direct proof.

\begin{cor}\label{cor:defo}
Let $X$ be an IHS manifold of type $\KKK^{[2]}$ with a nonsymplectic involution $\sigma$.
The following assertions are equivalent:
\begin{enumerate}
\item $\rank\Inv(\sigma)=1$.
\item $(X,\sigma)$ is deformation equivalent to a Beauville involution.
\item $(X,\sigma)$ is deformation equivalent to an O'Grady involution.
\end{enumerate}
\end{cor}

\section{New geometric constructions of nonsymplectic involutions}\label{s:new}

\subsection{Double Beauville involutions}

Let $S$ be a complex projective K3 surface of Picard number $1$, with a very ample polarization $H$ of square $H^2=2t$, ${t\geq 2}$. Using the Global Torelli theorem (see Theorem~\ref{th:torelli}), Boissi\`ere--Cattaneo--Nieper-Wi{\ss}kirchen--Sarti~\cite[Theorem 1.1]{bcns} proved that the Hilbert square~$S^{[2]}$ admits a nontrivial automorphism if and only if the following arithmetic conditions are satisfied:
\begin{itemize}
\item $t$ is not a square,
\item Pell's equation $P_t(-1)\colon x^2-ty^2=-1$ admits a solution,
\item Pell's equation $P_{4t}(5)\colon x^2-4ty^2=5$ has no solution.
\end{itemize}
In this case, $S^{[2]}$ admits a unique nontrivial automorphism, which is a nonsymplectic involution. Denote $h$ the class of $H$ in the splitting $\NS(S^{[2]})\cong \NS(S)\oplus \IZ\delta$ (see~Section~\ref{ss:beauville}) and by $(a,b)$ the minimal positive solution of $P_t(-1)$. The involution acts on $\HH^2(S^{[2]},\IZ)$ as the reflection in the class $D=bh-a\delta$, which is ample of square~$2$, so this gives a point in the moduli space~$\cM_{\langle 2 \rangle}^\rho$. However it is hard to produce a geometric realization of this involution if $t\neq 2$, this is the goal of this section.

It is easy to check that the arithmetic assumptions above are satisfied in particular if $t=(2\alpha+1)^2+1$, with $\alpha\geq 1$: the minimal solution of $P_t(-1)$ is $(2\alpha+1,1)$ and $P_{4t}(5)$ has no solution modulo $8$. We denote by $\Sigma_\alpha$ any K3 surface of Picard number $1$ and polarization of square $(2\alpha+1)^2+1$, by $\sigma_\alpha$ the nonsymplectic involution on $\Sigma_\alpha^{[2]}$ and by $L_\alpha\coloneqq\IZ h_1\oplus\IZ h_2$ the rank two lattice with Gram matrix 
$$
\begin{pmatrix} 4 & 4+2\alpha\\ 4+2\alpha & 4\end{pmatrix}.
$$

\begin{theorem} For any $\alpha\geq $1, the pair $(\Sigma_\alpha^{[2]},\sigma_\alpha)$   deforms to the pair $(S_\alpha^{[2]},\kappa_\alpha)$ where $S_\alpha$ is a K3 surface with Picard lattice $L_\alpha$ and $\kappa_\alpha\coloneqq\sigma^1_\alpha\sigma^2_\alpha\sigma^1_\alpha$, where $\sigma^1_\alpha,\sigma^2_\alpha$ are two Beauville involutions on $S_\alpha^{[2]}$. The deformation follows a deformation path of polarized K3 surfaces from $\Sigma_\alpha$ to $S_\alpha$.
\end{theorem}

\begin{proof}
The lattice $L_\alpha$ is even and hyperbolic of rank $2$, so by Morrison~\cite[Corollary \ 2.9]{morrison} it admits a primitive embedding in the K3 lattice and there exists a projective K3 surface $S_\alpha$ such that $\NS (S_\alpha) \cong L_\alpha$.
For any  $d = x h_1 + y h_2$, $x,y\in \IZ$, we have 
$$
d^2 = 4\left( x^2 + ( 2 + \alpha) xy + y^2 \right)=(4+\alpha)(x+y)^2-\alpha(x-y)^2,
$$ 
so there are no $(-2)$-curves on $S_\alpha$. It follows by the Nakai--Moishezon criterion for ampleness (see \cite[Proposition \ 1.4]{huybrechts}) that the ample cone of $S_\alpha$ coincides with its positive cone. Recall that the positive cone is the connected component of the cone $\{d\in L_\alpha\otimes_\IZ\IR\,|\, d^2>0\}$  which contains the K\"ahler classes. Since $h_1^2>0$ and $h_2^2>0$, the first quadrant $\IZ_{>0}h_1+\IZ_{>0}h_2$ and its symmetric by the origin are contained in this cone, and without loss of generality we may assume that the K\"ahler cone intersects the first quadrant. It follows that the positive cone is given by the inequalities $y>\frac{1+\epsilon\beta}{\epsilon\beta-1}x$ with $\epsilon=\pm 1$ and $\beta=\sqrt{\frac{\alpha}{4+\alpha}}$.

\begin{lemma}\label{lem:very_ample}
The classes $h_1,h_2,(2+2\alpha)h_1-h_2$ and $(2+2\alpha)h_2-h_1$ are very ample.
\end{lemma}

\begin{proof}[Proof of Lemma~\ref{lem:very_ample}]
These four classes are clearly in the ample cone. Their associated linear systems have no base components: by \cite[\S3.8, Theorem (d)]{reid} if any of these linear systems  has  base components, then it decomposes as $aE+\Gamma$ where $a$~is an integer, $|E|$~is a free pencil and $\Gamma$~is a $(-2)$-curve such that $E\cdot\Gamma=1$, but $S_\alpha$ contains  no $(-2)$-class. These linear systems thus define regular maps 
$$
\varphi_{|h_1|}, \varphi_{|h_2|}\colon S_\alpha \longrightarrow \IP^3, \qquad \varphi_{|(2 + 2\alpha) h_1 - h_2|}, \varphi_{|(2 + 2\alpha) h_2 - h_1 |}\colon S_\alpha \longrightarrow \IP^{1+t}.
$$
Let $d$ be any of the four primitive divisors in issue. We show that $d$ is not hyperelliptic by using Saint-Donat's criterion for determining hyperelliptic divisors \cite[Theorem\ 5.2]{saintdonat} : for this it is easy to check,  by reduction modulo $2$, that there is no class $E = x h_1 + y h_2$ such that $E^2 = 0$ and $E \cdot d = 2$. It follows that the regular map $\varphi_{|d|}$ is birational onto its image, but since $S_\alpha$ contains no $(-2)$-curve it is an embedding (see~\cite[(4.2) and \S6]{saintdonat}: by the genus formula, any contracted curve has square $-2$).
\end{proof}

The maps $\varphi_{|h_i|}$ embed $S_\alpha$ in $\IP^3$ in two different ways as a quartic with no line, so $S_\alpha^{[2]}$~has two different Beauville involutions $\sigma^i_{\alpha}$, $i=1,2$.
The N\'eron--Severi lattice of $S_\alpha^{[2]}$ is isometric to $\NS (S_\alpha) \oplus \langle -2 \rangle$ and $\sigma^i_{\alpha}$ acts on $\HH^2(S_\alpha^{[2]}, \IZ)$ as the reflection in the class $h_i - \delta$.
Consider the nonsymplectic involution $\kappa_{\alpha} \coloneqq \sigma^1_{\alpha} \sigma^2_{\alpha} \sigma^1_{\alpha}$. An easy computation shows that it acts on $\HH^2(S_\alpha^{[2]}, \IZ)$ as the
reflection in the class
$$
D_{\alpha}\coloneqq ((2 + 2\alpha)h_1 - h_2) - (2\alpha + 1)\delta,
$$ 
which has square $2$ and, by Theorem \ref{th:ample}, is ample and generates the invariant lattice of $\kappa_\alpha$.
By Theorem~\ref{th:defo}, the pairs $(\Sigma_\alpha^{[2]},\sigma_\alpha)$ and $(S_\alpha^{[2]},\kappa_\alpha)$ are deformation equivalent and deform to a Beauville involution, but observe that $\kappa_\alpha$ is not a Beauville involution since the class $D_\alpha$ cannot be written as $[H']-\delta$ with $H'$ very ample of square $4$ on $S_\alpha$. The pairs $(\Sigma_\alpha,H)$ and $(S_\alpha,(2+2\alpha)h_1-h_2)$ belong to the moduli space of ample $\langle 2t\rangle$-polarized K3 surfaces and $S_\alpha$ is obtained from $\Sigma_\alpha$ by moving the period along a path which increases the N\'eron--Severi group from $\IZ h$ to $\IZ h_1\oplus \IZ h_2$.
\end{proof}

\begin{rem}\text{}
\begin{enumerate}
\item We obtain a significatively more accurate result that in Theorem~\ref{th:defo} since the deformation path is here explicit. When deforming $S_\alpha$ to $\Sigma_\alpha$, the Beauville involutions $\sigma_\alpha^1,\sigma_\alpha^2$ disappear but $\kappa_\alpha$ survives.
\item By Oguiso~\cite[Theorem 4.1]{oguiso}, if $\alpha$ is even and $\alpha>8$, the two quartic K3 surfaces $\varphi_{|h_1|}(S_\alpha)$ and $\varphi_{|h_2|}(S_\alpha)$ are isomorphic as abstract varieties, but there is no birational transformation of $\IP^3$ sending one to the other.
\end{enumerate}
\end{rem}

\subsection{Nodal K3 surfaces}\label{ss:nodesK3}
In this section we give a new geometric construction of nonsymplectic involutions on the Hilbert square of a K3 surface, using K3 surfaces with nodes and high Picard number. 

\subsubsection{K3 surface with one node}

Let $\widetilde S$ be a general nodal K3 surface in $\IP^4$ ($\widetilde S$~is the complete intersection of a quadric and a cubic) and $\beta\colon S\to \widetilde S$~its minimal K3 resolution. We denote by $\widetilde H$ the hyperplane section on $\widetilde S$, $H\coloneqq\beta^\ast \widetilde H$  and by $\varepsilon$ the $(-2)$-exceptional curve. We thus have
$$
\NS(S)=\IZ\widetilde H\oplus\IZ\varepsilon\cong\langle 6\rangle\oplus\langle -2\rangle.
$$
We define a birational involution $\sigma$ on $S^{[2]}$ as follows. We denote by~$p$ the node of~$\widetilde S$. Take two general points $q_1,q_2\in \widetilde S$, distinct from $p$. The family of hyperplanes through $p,q_1,q_2$ is a pencil since the general points $q_1,q_2$ impose independent conditions on the linear system $|\widetilde H|$, which is $4$-dimensional. Let $H_1,H_2$ be any two generators of this pencil, then $H_1$~cuts $\widetilde S$ along a degree six curve singular at $p$, which $H_2$~cuts in two more points $q_3,q_4$ (generically distinct). So $\{p,q_1,q_2,q_3,q_4\}$ is the base locus of the pencil. We thus define a birational involution $\sigma$ on $S^{[2]}$ by sending $\{q_1,q_2\}$ to $\{q_3,q_4\}$.

\begin{prop}\label{prop:onenode}
The involution $\sigma$ is a Beauville involution on $S^{[2]}$.
\end{prop}

\begin{proof}
The linear system $|H-\varepsilon|$ on $S$ has no base components since $H$ is very ample. We see as in the proof of  Lemma~\ref{lem:very_ample} that the divisor $H-\varepsilon$ is not hyperelliptic, hence the regular map $\varphi_{|H-\varepsilon|}\colon S\to\IP^3$ is birational onto its image. This map contracts no $(-2)$-curve, otherwise there would exist a class $\alpha H+\beta \varepsilon$  such that $(\alpha H +\beta \varepsilon)(H-\varepsilon)=\beta+3\alpha=0$ and $(\alpha H+\beta \varepsilon)^2=6\alpha^2-2\beta^2=-2$, which is impossible. So $\varphi_{|H-\varepsilon|}$ embeds $S$ as a quartic $\Sigma$ in $\IP^3$, and it contains no line otherwise a $(-2)$-class $\alpha H+\beta\varepsilon$ would be sent to a line, so we would get $1=(\alpha H +\beta \varepsilon)(H-\varepsilon)=6\alpha+2\beta$ which has clearly no solution. 

The hyperplane sections on $\widetilde S$ passing through the node $p$ correspond to divisors in the system $|H-\varepsilon|$ on $S$, so the hyperplanes $H_1,H_2$ are sent to hyperplane sections $h_1,h_2$ of $\Sigma$ which contain the images of the points $q_1,q_2,q_3,q_4$. These four points thus lie on the line $h_1\cap h_2$: this shows that the birational involution $\sigma$ on~$S^{[2]}$ is nothing else than the Beauville involution on $\Sigma^{[2]}$.
\end{proof}

\begin{rem} It follows that the invariant lattice of $\sigma$  on $S^{[2]}$ is generated by the square $2$ ample divisor $(H-\varepsilon)-\delta$. By Bini~\cite{gil} we have $\Aut(S)=\{\id\}$, but $\Aut(S^{[2]})$  contains at least an involution. We refer to~\cite{bcns} for a similar property on K3 surfaces of Picard rank one.
\end{rem}

\subsection{K3 surfaces with several nodes}
Consider the even hyperbolic lattice
$$
R_k = \langle 4 + 2k \rangle \oplus \bigoplus_{i = 1}^k \langle -2 \rangle, 
\,\, k\leq 10.
$$
By \cite[Corollary 2.9, Remark 2.11]{morrison} there exists a K3 surface with 
N\'eron--Severi lattice  
isomorphic to $R_k$. For $k=1$ we recover the K3 surface with one node described above, for $k=2$ we have a K3 surface in $\IP^5$ with two nodes 
(a complete intersection of three quadrics). 

The condition on $k$ comes from the observation that if $k>10$, $R_k$ cannot be the N\'eron--Severi group of a K3 surface: indeed in this situation 
 the N\'eron--Severi lattice has rank at least $12$ and discriminant group of length $k+1$, so that the rank of the 
transcendental lattice is at most $10$ with a discriminant group  again of length $k+1$, which is impossible. If $k>10$, the lattice $R_k$ can only be a sublattice of the N\'eron--Severi group, see~\cite[Theorem 2.7]{rmj}
for some examples when $k=16$. 

Denote by $H_k$ the generator of the summand $\langle 4 + 2k \rangle$ and by $\varepsilon_i$ the generator of the $i$-th copy of the lattice $\langle -2 \rangle$, with $i=1,\ldots, k$. The lattice $R_k$ is the N\'eron--Severi group of the K3 surface $S_k$ obtained as the minimal resolution of a K3 surface $\widetilde S_k$ embedded in $\IP^{k + 3}$ as a surface of degree $4 + 2k$, with $k$ singularities of type $A_1$ at points $p_1, \ldots, p_k$. The curves $\varepsilon_i$, $i=1,\ldots, k$, correspond to the exceptional divisors obtained by the blowup of the singular points. We define similarly as above a birational involution on $S_k^{[2]}$.
Take two general points $q_1,q_2\in \widetilde S_k$ distinct from $p_1,\ldots,p_k$. The family of hyperplanes through $p_1,\ldots,p_k,q_1,q_2$ has dimension $(k+3)-(k+2)=1$. Let $H_1,H_2$ be any two generators of this pencil. Then $H_1$~cuts $\widetilde S_k$ along a degree $4+2k$ curve with nodes at the points $p_1,\ldots,p_k$. The divisor~$H_2$ cuts this curve twice in the singular points $p_1,\ldots,p_k$ and once at $q_1,q_2$ so it cuts the curve  in two other points $q_3,q_4$ (generically distinct). We thus define a birational involution~$\sigma$ on~$S_k^{[2]}$ by sending $\{q_1,q_2\}$ to $\{q_3,q_4\}$.

\begin{prop}
The involution $\sigma$ is a Beauville involution on $S_k ^{[2]}$.
\end{prop}

\begin{proof}
The proof is similar to those of Proposition~\ref{prop:onenode}.
The divisor $H_k - \sum_{i = 1}^k \varepsilon_i$ has square $4$, its associated linear system has no base components, it is not hyperelliptic and it contracts no $(-2)$-curve. So it embeds $S_k$ as a quartic $\Sigma_k$ in $\IP^3$ which contains no line: since any two divisors of the lattice $R_k$ have even intersection number, no $(-2)$-curve is sent to a line. 

\end{proof}

\bibliographystyle{amsplain}
\bibliography{Biblio}

\end{document}